\documentclass[reqno,table]{amsart}
\usepackage{amssymb,amsmath,amsfonts,amsthm,epsfig,amscd,stmaryrd}
 \usepackage[textsize=tiny]{todonotes}
\usepackage{tikz,tikz-cd}
\usetikzlibrary{arrows, matrix,arrows.meta, positioning,decorations.pathmorphing}
\usepackage[bookmarksopen,bookmarksdepth=2]{hyperref}

\newtheorem{cor}[subsection]{Corollary}
\newtheorem{thm}[subsection]{Theorem}

\newtheorem{lem}[subsection]{Lemma}
\newtheorem{iprob}{Problem}

\newtheorem{ithm}[iprob]{Theorem}

\theoremstyle{definition}

\theoremstyle{definition}
\newtheorem{remark}[subsection]{Remark}

\newcommand{\Ind}{\operatorname{Ind}}
\newcommand{\into}{\hookrightarrow}
\newcommand{\cO}{\mathcal{O}}
\newcommand{\cG}{\mathcal{G}}
\newcommand{\Qp}{\Q_p}
\newcommand{\Fbar}{\overline{\F}}
\newcommand{\Qbar}{\overline{\Q}}
\newcommand{\Qpbar}{\Qbar_p}
\newcommand{\A}{\mathbf{A}}
\newcommand{\Fp}{\F_p}
\newcommand{\Fpbar}{\Fbar_p}
\newcommand{\psibar}{\overline{\psi}}
\def\OO{\mathrm{O}}

\def\Zhat{\widehat{\Z}}
\def\sss{\mathrm{ss}}

\def\cusp{\mathrm{cusp}}
\def\AAA{\mathcal{A}}

\def\numequation{\addtocounter{subsubsection}{1}\begin{equation}}
\def\numalign{\addtocounter{subsubsection}{1}\begin{align}}
\def\nummultline{\addtocounter{subsubsection}{1}\begin{multline}}
\def\anumequation{\addtocounter{subsection}{1}\begin{equation}}
\def\anummultline{\addtocounter{subsection}{1}\begin{multline}}

\DeclareMathOperator{\Gal}{Gal}
\DeclareMathOperator{\GL}{GL}
\DeclareMathOperator{\SL}{SL}
\DeclareMathOperator{\Sp}{Sp}
\DeclareMathOperator{\Sym}{Sym}
\DeclareMathOperator{\GO}{GO}

\DeclareMathOperator{\SO}{SO}
\DeclareMathOperator{\GSp}{GSp}

\def\Q{\mathbf{Q}}
\def\Z{\mathbf{Z}}
\def\C{\mathbf{C}}
\def\F{\mathbf{F}}
\def\R{\mathbf{R}}
\def\rbar{\overline{r}}

\def\sbar{\overline{s}}

\newcommand{\rhobar}{\overline{\rho}}
\newcommand{\varepsilonbar}{\overline{\varepsilon}}

\title{Cuspidal cohomology classes for~$\GL_n(\Z)$}

\author[G.~Boxer]{George Boxer}  \email{g.boxer@imperial.ac.uk} \address{Department of
  Mathematics, Imperial College London,
  London SW7 2AZ,~UK}

\author[F.~Calegari]{Frank Calegari}  \email{fcale@math.uchicago.edu} \address{The University of Chicago,
5734 S University Ave,
Chicago, IL 60637, USA}

\author[T.~Gee]{Toby Gee} \email{toby.gee@imperial.ac.uk} \address{Department of
  Mathematics, Imperial College London,
  London SW7 2AZ,~UK}

  \thanks{G.B.\ was supported by a Royal Society University Research Fellowship}
  
 \thanks{F.C. \ was supported in part by NSF Grant DMS-2001097.}

 \thanks{ T.G.\ was supported in part by an ERC Advanced grant. This
   project has received funding from the European Research Council
   (ERC) under the European Union’s Horizon 2020 research and
   innovation programme (grant agreement No. 884596)}

\begin{document}

\begin{abstract} We prove the existence of a
cuspidal automorphic representation~$\pi$ for~$\GL_{79}/\Q$ 
of level one and weight zero. We construct~$\pi$ using symmetric power functoriality
and a change of weight theorem, using Galois deformation theory.
As a corollary, we construct
the first known cuspidal cohomology classes in $H^*(\GL_{n}(\Z),\C)$
for any~$n > 1$.
\end{abstract}

\dedicatory{To Laurent Clozel, in admiration.}

\maketitle
\section{Introduction}

It is a well-known fact that there do not exist any cuspidal modular
forms of level~$N=1$ 
and weight~$k = 2$. From the Eichler--Shimura isomorphism,
this is equivalent to the vanishing of the cuspidal cohomology groups 
$$H^i_{\mathrm{cusp}}(\GL_2(\Z),\C)=0$$
for all~$i$ (particularly~$i=1$).
 It is natural to wonder what happens in higher rank.
 
 \begin{iprob}\label{isitopen}
Does there exist an~$n > 1$ such that
$H^i_{\mathrm{cusp}}(\GL_n(\Z),\C) \ne 0$ for some~$i$?
\end{iprob}

Higher rank analogues of the Eichler--Shimura isomorphism (see  Remark~\ref{rem:GLn-Matsushima})
imply 
that Problem~\ref{isitopen}  is equivalent to the existence of cuspidal automorphic representations~$\pi$ for~$\GL_n/\Q$ which have level one and weight zero.  Here level one means that $\pi_p$ is unramified for all primes $p$ and weight zero means that $\pi_\infty$ has the same infinitesimal character as the trivial representation.
The %
work of Fermigier and subsequently of Miller (\cite[Cor.\ 1]{Fermigier}
for $n\le 23$, \cite[Thm.\ 1.6]{Miller} for $n<27$) showed that 
there are no such~$\pi$ for
all~$1 < n < 27$; their methods are analytic and are related to the
Stark--Odlyzko positivity technique~\cite{Odlyzko} for lower bounds on
discriminants of number
fields. (In fact \cite{Fermigier} and \cite{Miller} 
formulate their main results in terms of the vanishing of the cuspidal cohomology
of~$\GL_n(\Z)$ and~$\SL_n(\Z)$ respectively; for the equivalence of
these statements with each other and with the non-existence of
such~$\pi$, see Remark~\ref{rem:GLn-Matsushima}.) %

Problem~\ref{isitopen} has subsequently been raised explicitly by a
number of people, including~\cite[\S2.5]{MR3642468}, \cite{KhareSoft},
and \cite[\S1.2]{ChevRenard}, where it is referred to as a
``well-known'' problem.  One motivation for this question, emphasized
by Khare, is that the vanishing of the
$H^i_{\mathrm{cusp}}(\GL_n(\Z),\C)$ for a given~$n$ could provide the
base case for an inductive proof of the analogue of Serre's conjecture
in dimension~$n$.  It was unclear to many people (including some of
the authors of this paper) whether it was reasonable to hope for this
vanishing for all~$n$, although in recent years the work of~Chenevier
and~ Ta\"{\i}bi on self-dual automorphic representations of level~$1$
(see e.g.\ the introduction to \cite{Chevtwo}) had made this seem
unlikely.  Another reason to expect an affirmative answer to
Problem~\ref{isitopen} is by comparison to the aforementioned
discriminant bounds of Odlyzko, which for a number field~$K/\Q$ give
positive constant lower bounds for the root
discriminant~$\delta_K =|\Delta_K|^{1/[K:\Q]}$ as the degree of~$K$
tends to infinity. One may ask whether there might exist a lower
bound which tended to infinity in~$[K:\Q]$. The answer to this
question is no by the Golod--Shafarevich construction; the existence
of class field towers gives an infinite sequence of fields of
increasing degree such that~$\delta_K$ is constant.

Our main theorem resolves Problem~\ref{isitopen} in the affirmative:
\begin{ithm}[Theorem~\ref{thm: p 107 results}, Corollary~\ref{cor:79}] \label{main} There exist cuspidal automorphic
representations for~$\GL_n/\Q$ of level one and weight zero
for~$n=79$, $n=105$, and~$n=106$.
In particular, $H^*_{\cusp}(\GL_n(\Z),\C) \ne 0$ for these~$n$.
\end{ithm}
Our argument works for other values of~$n$ 
(presumably infinitely
many,
 although we do not know how to prove this;
see Remark~\ref{rem: expect positive density irreducible}). 
In light of Theorem~\ref{main}, there is the obvious variation of
Problem~\ref{isitopen}:
 \begin{iprob}\label{itisopen}
 What is the smallest~$n > 1$ such that
$H^i_{\mathrm{cusp}}(\GL_n(\Z),\C) \ne 0$ for some~$i$?
\end{iprob}
We know from~\cite{Miller} and Theorem~\ref{main}
 that the answer satisfies~$27 \le n \le 79$. 
The work of Chenevier and Ta\"{\i}bi ~\cite{Chevtwo} suggests that the real 
answer is much closer to the lower bound than the upper bound.

\medskip

While the formulation of Problem~\ref{isitopen} makes no reference to
motives or Galois representations, according to standard conjectures in the Langlands program it is 
equivalent to the existence of irreducible rank~$n$ pure motives (with coefficients)
over~$\Q$ with everywhere good reduction and Hodge numbers
$0,1,\dots,n-1$, or to the existence of irreducible Galois representations $\rho:G_\Q\to\GL_n(\overline{\Q}_p)$ unramified away from $p$ and crystalline with Hodge--Tate weights $0,1,\ldots,n-1$ at $p$.  In fact, we will proceed by producing such Galois representations. %

Our approach to proving Theorem~\ref{main} is ultimately based on the
conjecture of Serre~\cite{MR885783} predicting the existence of congruences between
modular forms of different weights.  If $f$ is a cuspidal eigenform of level 1 and weight $k$ and the mod $p$ Galois representation $\rhobar_{f,p}:G_\Q\to\GL_2(\overline{\F}_p)$ is irreducible, then Serre predicts that there exists a modular form $g$ of weight 2 and level 1 with $\rhobar_{g,p}\simeq\rhobar_{f,p}$ if and only if $\rhobar_{f,p}|_{G_{\Q_p}}$ admits a crystalline lift with Hodge--Tate weights $0$ and~$1$.  Of course this cannot actually occur as no such $g$ exists!
The natural generalization of Serre's conjecture for larger $n$
predicts that if $\pi$ is a regular algebraic essentially self dual
cuspidal automorphic representation for $\GL_n/\Q$ of level 1 and
arbitrary weight, and the mod $p$ Galois representation
$\rhobar_{\pi,p}:G_\Q\to\GL_n(\overline{\F}_p)$ has ``large'' image,
then there exists an essentially self dual $\pi'$ of level 1 and weight 0 with $\rhobar_{\pi',p}\simeq\rhobar_{\pi,p}$ if and only if $\rhobar_{\pi,p}|_{G_{\Q_p}}$ admits a crystalline lift
which is either symplectic or orthogonal (depending on~$\pi$) up to twist, and 
with Hodge--Tate weights $0$, $1,\ldots,n-1$.  In many instances, these ``change of weight'' congruences may in fact be produced using automorphy lifting theorems and the Khare--Wintenberger method, as in \cite{gee051,gg,BLGGT}.

It remains to explain how we find the $\pi$ to which the above strategy can be applied.  For this, we need a supply of $\pi$ for which $\rhobar_{\pi,p}|_{G_{\Q_p}}$ may be readily understood.  Our idea is to take $\pi$ to be $\Sym^{n-1}f$ (up to twist) for $f$ a modular form of level~$1$; this symmetric power lift is now available thanks to the recent work of Newton--Thorne (see~\cite[Thm.\ A]{symmetric} for the version we use).  If $f$ is a cuspidal eigenform of level 1 and weight $k<p$, then typically $f$ will be ordinary at $p$ and the Galois representation $\rhobar_{f,p}|_{I_p}$ will be a nonsplit extension of $\varepsilonbar^{1-k}$ by $1$, where $\varepsilonbar$ denotes the mod $p$ cyclotomic character.  In this case no twist of $\Sym^{n-1}\rhobar_{f,p}|_{G_{\Q_p}}$ will have a crystalline lift of Hodge--Tate weights $0,\ldots,n-1$, at least for $n\leq p$.  On the other hand in the less typical situation that $\rhobar_f|_{G_{\Q_p}}$ is semisimple (or equivalently tamely ramified) we are sometimes able to succeed.  Here there are two possibilities, either $f$ is still ordinary at $p$ but the extension splits and $\rhobar_{f,p}|_{G_{\Q_p}}$ is a sum of two characters, or $f$ is non-ordinary at $p$ and $\rhobar_{f,p}|_{G_{\Q_p}}$ is irreducible.

As an illustration, if $f$ is ordinary at $p$, $\rhobar_{f,p}|_{G_{\Q_p}}$ splits, and $(k-1,p-1)=1$, then as $\varepsilonbar$ has order $p-1$, we find that
$$\Sym^{p-2}\rhobar_{f,p}|_{I_p}=\Sym^{p-2}(1\oplus\varepsilonbar^{1-k})=\bigoplus_{i=0}^{p-2}\varepsilonbar^{i(1-k)}=\bigoplus_{i=0}^{p-2}\varepsilonbar^i,$$
and hence $\Sym^{p-2}\rhobar_{f,p}|_{G_{\Q_p}}$ has a crystalline lift
of Hodge--Tate weights $0,1,\ldots,p-2$ which on inertia is simply a
sum of powers of the cyclotomic character.  This leads to the case
$n=106$ of theorem, taking $f$ to be the cusp form of level 1 and
weight 26 and $p=107$, while the case $n=105$ comes from a similar
consideration of $\Sym^{104}f$. Our ``change of weight'' theorem is
proved by extending the techniques introduced in ~\cite{gee051} and
developed further by Gee and Geraghty in~\cite{gg}, combining the
Khare--Wintenberger method with automorphy lifting theorems for Hida
families on unitary groups due to Geraghty~\cite{ger} (and refined by
Thorne~\cite{jack}). The case $n=79$ comes from considering
$\Sym^{78}f$ for a modular form $f$ which is non-ordinary at
$p=79$. Here the change of weight theorem is more involved and closer
to the arguments of~\cite{BLGGT}, using the Harris tensor product
trick.

  \begin{remark}
    \label{rem: expect to get multiples} While we expect that a
    cuspidal  automorphic representation of~$\GL_{n}$ of level one and
    weight zero should exist for all sufficiently large~$n$, we do not
    know how to prove this, even conditionally on Langlands
    functoriality. We can however give such a conditional argument for
    the existence for infinitely many~$n$. Indeed, if~$\pi$ is cuspidal automorphic of level one and weight zero
    for~$\GL_n/\Q$ with~$n$ odd, then for each $m\ge 1$, there is
    conjecturally a cuspidal automorphic representation of level one
    and weight zero for~$\GL_{nm}/\Q$. Indeed, for each~level one
    cuspidal eigenform~$f$ of weight~$n+1$ (such an~$f$ exists
    because~$n > 26$), the conjectural tensor
    product~$\pi \boxtimes \Sym^{m-1} f$ should be automorphic and
    cuspidal of level one and weight zero.
  \end{remark}

\begin{remark}
  \label{rem:GLn-Matsushima}Encouraged by one of the referees, we
now make some  clarifying remarks about 
the cuspidal cohomology groups
 of~$\GL_{n}(\Z)$ and~$\SL_n(\Z)$ and 
 their relationship with cuspidal automorphic
representations. A precise statement is as follows:
\begin{enumerate}
  \item If~$n$ is odd, then $H^*_{\cusp}(\GL_n(\Z),\C)=H^*_{\cusp}(\SL_n(\Z),\C)$. 
\item If~$n$ is even, then
  $\dim H^*_{\cusp}(\SL_n(\Z),\C)=2\dim H^*_{\cusp}(\GL_n(\Z),\C)$.
\item In either case, the spaces are nonzero if and only if there exists a cuspidal automorphic representation for~$\GL_n/\Q$
of level one and weight zero.
\end{enumerate}

 We now recall some definitions.
 There are two  locally symmetric spaces of level one associated to~$\GL_n/\Q$
  given by the quotients \[\GL_n(\Q) \backslash \GL_n(\A)
    /K_\infty\GL_n(\Zhat)\] where~$K_\infty$ is either~$\OO(n)$ or~$\SO(n)$. These can be identified with
    the quotient of the connected contractible symmetric space~$\GL_n(\R)/\OO(n)$ by~$\GL_n(\Z)$
    and~$\SL_n(\Z)$ respectively. When $n$ is odd, we have ~$\GL_n(\Z)
    =\SL_n(\Z) \times \{\pm 1\}$ and these two spaces are equal, but
    when~$n$ is even, one is a double cover of the other. The cuspidal
    cohomology groups $H_{\cusp}^*(\GL_n(\Z),\C)$ and
    $H_{\cusp}^*(\SL_n(\Z),\C)$ are defined to be the subspaces of the
    cohomology groups of the symmetric spaces of classes represented
    by harmonic cusp forms, and (see for example \cite[\S3.3]{HarderRag}
and \cite[Lem.\ 3.15]{MR1044819}) we have a commutative diagram as
    follows:
  \[\begin{tikzcd}
  \bigoplus_{\pi} H^{*}(\mathfrak{sl}_n,\OO(n); \pi_{\infty}) &H_{\cusp}^*(\GL_n(\Z),\C)  \\
 \bigoplus_{\pi} H^{*}(\mathfrak{sl}_n,\SO(n); \pi_{\infty}) & H_{\cusp}^*(\SL_n(\Z),\C),
	\arrow[hook, from=1-1, to=2-1]
	\arrow[hook, from=1-2, to=2-2]
        	\arrow["\simeq", from=1-1, to=1-2]
        	        \arrow["\simeq", from=2-1, to=2-2]
\end{tikzcd}\]
where the sums on the left hand side range over  the cuspidal automorphic representations~$\pi$ for~$\GL_n/\Q$
of level one and weight zero, and the cohomology is the relative Lie
algebra cohomology~\cite[I.5]{MR1721403}; it is immediate from the
definition that we have $H^{*}(\mathfrak{sl}_n,\OO(n);
\pi_{\infty})=H^{*}(\mathfrak{sl}_n,\SO(n);
\pi_{\infty})^{\OO(n)/\SO(n)}$.

 If~$n$ is even, there is a unique tempered %
 cohomological~$\pi_{\infty}$, and
  the~$(\mathfrak{gl}_n,\SO(n))$-cohomology is free as a $\C[\OO(n)/\SO(n)] \simeq \C[\Z/2\Z]$-module.
 (See~\cite[Lem.\ 3.14]{MR1044819} and its proof.) %
This results from the fact that the restriction of~$\pi_{\infty}$ to
the identity component~$\GL(\R)^{\circ}$ decomposes into the sum of
 two irreducible representations. In particular, there always
 exist~$\OO(n)/\SO(n)$-invariants; more precisely,
 $\dim H^*_{\cusp}(\SL_n(\Z),\C)=2\dim H^*_{\cusp}(\GL_n(\Z),\C)$,
 as claimed.

 If~$n$ is odd, the cohomologies
 of~$\GL_n(\Z)$ and~$\SL_n(\Z)$ agree as explained above.
  The interpretation for this in terms of the
 above diagram is as follows. There are now two tempered cohomological~$\pi_{\infty}$ which differ by a twist by the sign character of~$\GL_n(\R)$,
 and the action of~$\OO(n)/\SO(n)$
 on~$(\mathfrak{gl}_n,\SO(n))$-cohomology is either trivial or
 by~$-1$. (Again see~\cite[Lem.\ 3.14]{MR1044819}.)
 However, only the~$\pi_{\infty}$ with trivial central character can
 arise from a cuspidal automorphic representation of~$\GL_n$ with
 level one and weight zero, as the central character of such an
 automorphic representation is necessarily trivial. %
\end{remark}

\subsection{Acknowledgements}\label{subsec:acknowledgements}
We have been aware of Problem~\ref{isitopen} for some time, but it was
most recently brought to our attention at a lecture~\cite{CL} by
Ga\"{e}ten Chenevier at the conference \emph{Arithm\'{e}tique des
  formes automorphes} at Orsay in September, 2023, in honour of
Laurent Clozel's 70th birthday.  In light of this, together with the
obvious connections between the methods of this paper and Clozel's
work (Galois representations associated to self-dual automorphic
representations, modularity lifting theorems for self-dual Galois
representations, and symmetric power functoriality for modular forms,
to name but three), it is a pleasure to dedicate this paper to him.
We would also like to thank James Newton,  A.\ Raghuram, Will Sawin, 
Joachim Schwermer,
Olivier Ta\"{\i}bi and Jack
Thorne for helpful comments on earlier versions of this paper, and the
anonymous referees for several helpful comments and corrections.

\section{The ordinary case}\label{sec: ordinary}
We fix once and for all for each
prime~$p$ an isomorphism $\imath=\imath_p:\C\cong\Qpbar$, and we will
accordingly sometimes implicitly regard automorphic representations as
being defined over~$\Qpbar$, rather than~$\C$. In particularly we will
freely refer to ``the'' $p$-adic Galois representation associated to a
(regular algebraic) automorphic representation. We write~
$\rho_f:G_{\Q}\to\GL_2(\Qpbar)$ and $\rhobar_f:G_{\Q}\to\GL_2(\Fpbar)$
for the cohomologically normalized representations associated to an
eigenform~$f$. Let~$\varepsilon$ denote the~$p$-adic cyclotomic
character and~$\varepsilonbar$ its mod-$p$ reduction.

\begin{thm}\label{thm: congruence for f p and k}Let~$f$ be an
  eigenform of level~$\SL_2(\Z)$ and weight~$k\ge 2$, and
  let~$p>5$ 
  be a
  prime such that:
  \begin{enumerate}
  \item $\rhobar_f(G_{\Q})\supseteq \SL_2(\Fp)$.
  \item $(p-1,k-1)=1$.
  \item $f$ is ordinary at~$p$.
  \item $\rhobar_f|_{G_{\Q_p}}$ is semisimple.
  \end{enumerate}
  Then, for both~$n = p-1$ and~$n = p - 2$, there exists a self-dual
  cuspidal automorphic
  representation~$\pi$ for $\GL_{n}/\Q$ of level one and weight
  zero whose mod~$p$ Galois representation
  $\rhobar_{\pi}:G_{\Q}\to\GL_{n}(\Fpbar)$ is isomorphic to
  $$\Sym^{n-1}(\rhobar_{f} \otimes \varepsilonbar^{\frac{k-2}{2}})
  =   \varepsilonbar^{\frac{(n-1)(k-2)}{2}} \otimes  \Sym^{n-1} \rhobar_{f}.$$
  \end{thm} 
  \begin{proof}Let~$n = p-1$ or~$p-2$, and write~$G_n=\GSp_n$
  if~$n=p-1$ (equivalently, if $n$ is even), and $G_n=\GO_n$
  if~$n=p-2$ (equivalently, if~$n$ is odd). Let~$\F/\Fp$ be a finite extension such that
  $\rhobar_f(G_{\Q})\subseteq \GL_2(\F)$, and write
  \[\rhobar:=\Sym^{n-1}(\rhobar_{f} \otimes \varepsilonbar^{\frac{k-2}{2}})
  =   \varepsilonbar^{\frac{(n-1)(k-2)}{2}} \otimes  \Sym^{n-1}
  \rhobar_{f}  :G_{\Q}\to\GL_{n}(\F).\]
  Since~$\rho_f$ is symplectic with multiplier~$\varepsilon^{1-k}$,
the
twist~$\rhobar_{f} \otimes \varepsilon^{\frac{k-2}{2}}$ is symplectic
with multiplier~$\varepsilonbar^{-1}$, and so we can and do regard
$\rhobar$ as a representation $G_{\Q} \rightarrow G_{n}(\F)$ with
multiplier~$\varepsilonbar^{1-n}$.
   In particular, we have an isomorphism $ \rhobar \simeq  \rhobar^{\vee} \varepsilonbar^{1-n}.$

   By
  the hypotheses that~$f$ is ordinary at~$p$ and $\rhobar_{f} |_{G_{\Q_p}}$
  is semisimple, we can write
  \[      \rhobar_f|_{G_{\Qp}}\cong
    \psibar\oplus\psibar^{-1}\varepsilonbar^{1-k}  \] for some unramified
  character~$\psibar$, so that \[\rhobar|_{G_{\Qp}}\cong\bigoplus_{i=0}^{n-1}\psibar^{n-1-2i}\varepsilonbar^{(n-1)(k-2)/2-(k-1)i}.\]
Since~$(p-1,k-1)=1$, either $n=p-1$ or~$n=p-2$, and ~$\varepsilonbar$ has order~$(p-1)$,
it follows easily that there are 
  unramified characters ~$\psibar_i$ for $i=0,\dots,n-1$ such that
  \numequation
   \label{eqn:rhobar at p}
 \rhobar|_{G_{\Qp}}\cong
    \bigoplus_{i=0}^{n-1}\psibar_i\varepsilonbar^{-i}; \quad
    \psibar_{n-1-i}=\psibar_i^{-1}.
\end{equation}

  Since $\SL_2(\Fp)\subseteq \rhobar_f(G_{\Q}) $, the
  representation~$\rhobar$ is absolutely irreducible (see also
  Lemma~\ref{lem: adequacy}.) 
  Let~$E/\Qp$ be a finite extension with
  ring of integers~$\cO$ and residue field~$\F$. Recall that~$G_n=\GSp_n$
  if~$n$ is even, and $G_n=\GO_n$ if~$n$ is odd. 
  Write~$R$ for the complete local Noetherian $\cO$-algebra which is
  the universal deformation ring for~$G_n$-valued deformations
  of~$\rhobar$ which have multiplier~$\varepsilon^{1-n}$, are
  unramified outside~$p$, and whose restrictions to~$G_{\Qp}$ are
  crystalline and ordinary with Hodge--Tate weights
  $0,1,\dots,n-1$. 

  By~\cite[Prop.\ 4.2.6]{Bellovin}, every irreducible component of~
  $R$ has Krull dimension at least~$1$. (We are applying~\cite[Prop.\
  4.2.6]{Bellovin} with~$l$ equal to our~$p$, and the local deformation ring~$\overline{R}_p$
  being the union of those
  irreducible components of the corresponding crystalline deformation ring
  which are ordinary, as in~\cite[Lem.\ B.4]{MR4392460}; this is indeed a nonempty
  set of components because~\eqref{eqn:rhobar at p} shows that
  $\rhobar|_{G_{\Qp}}$ admits an
  ordinary crystalline lift, by lifting the
  characters~$\psibar_i$ to their Teichm\"uller lifts and
  the~$\varepsilonbar^{-i}$ to~$\varepsilon^{-i}$.  %
  In order to
  apply this proposition, we need to verify
  that~$H^0(\Q,(\mathfrak{g}^{0}_n)^*(1)) = 0$, where
  $\mathfrak{g}_n^{0}$ is respectively $\mathfrak{sp}_n$ or
  $\mathfrak{so}_n$ according to whether~$n$ is even or odd. To see
  this, it suffices to check that there are no invariants after taking
  the semi-simplification.
But~$(\mathfrak{g}^{0}_{n})^{*,\sss} \subset \mathfrak{gl}_{n}^{*,\sss} \simeq \mathfrak{gl}_{n}^{\sss}$ (such an inclusion
need not exist before taking semi-simplifications) and the latter module is isomorphic
to~$\bigoplus_{i=0}^{n-1} (\Sym^{2i} \rhobar_f)^{\sss} \otimes
\det(\rhobar_f)^{-i}\subset \bigoplus_{i=0}^{p-1} (\Sym^{2i} \rhobar_f)^{\sss} \otimes
\det(\rhobar_f)^{-i}$.
From the representation theory of~$\SL_2(\F)$, we see that 
the only characters occurring in each of these factors
occur with multiplicity at most one and only for~$i=0$, $2i=p+1$, and~$2i=2p-2$ (the second case only occurring
when~$\F = \F_p$).
 The characters that arise are in particular
self-dual, and so distinct from~$\varepsilonbar^{-1}$ since~$p > 3$. It
follows that~ $H^0(\Q,(\mathfrak{g}^{*,\sss}_n)^*(1)) \subseteq
H^0(\Q,\mathfrak{gl}_{n}^{\sss}(1)) = 0$, as required.
 The remaining hypotheses of
  ~\cite[Prop.\ 4.2.6]{Bellovin} hold because %
  the multiplier character $\varepsilon^{1-n}$ is odd/even precisely when $G_n$ is
  symplectic/orthogonal, and the
  Hodge--Tate weights $0,1,\dots,n-1$ are pairwise distinct.)

  Let~$F/\Q$ be an imaginary quadratic field in which~$p$ splits and which is disjoint from
  $(\Qbar)^{\ker\rhobar}(\zeta_{p})$. As in~\cite{cht} we let~$\cG_n$
 denote the semi-direct product of $\cG_n^0=\GL_n \times \GL_1$ by the group $\{1 , \jmath\}$
where 
\[ \jmath (g,a) \jmath^{-1}=(ag^{-t},a), \]
with multiplier character  $\nu:\cG_n \to \GL_1$  sending $(g,a)$ to
$a$ and $\jmath$ to $-1$. Following~\cite[\S 1.1]{BLGGT}, given  a homomorphism
$\psi:G_{\Q}\to G_n(R)$, we have an associated homomorphism
$r_{\psi}:G_{\Q}\to\cG_n(R)$, whose multiplier character is that 
of~$r$ multiplied by~$\delta_{F/\Q}^n$, where~$\delta_{F/\Q}$ is the
quadratic character corresponding to the
extension~$F/\Q$. Explicitly, if ~$A_n$ is the matrix defining the
pairing for the group~$G_n$ (so $A_n=1_n$ if~$n$ is odd and
$A_n=J_n$ if~$n$ is even, where~$J_n$ is the standard symplectic
form), then  $r_{\psi}$ can be defined as the
composite \[G_{\Q}\stackrel{\psi\times \operatorname{pr}}{\longrightarrow} G_n(R)\times
  G_{\Q}/G_F\to\cG_n(R),\] where $\operatorname{pr}$ is the projection
$G_{\Q}\to G_{\Q}/G_F\cong\{\pm 1\}$, and
the second map is the injection 
\numequation\label{eqn:Gn-to-cGn}G_n\times\{\pm 1\}\into\cG_n\end{equation} given by \begin{align*}
                                        r((g,1))& = (g,\nu(g)),\\
                                        r((g,-1))& = (g,\nu(g))\cdot(A_{n}^{-1},(-1)^{n+1}) \jmath.
\end{align*}
In particular
  we can apply this construction to~$\rhobar$, and we write $\rbar:=r_{\rhobar}:G_{\Q}\to\cG_n(\F)$.

  We let $R_F$ be the complete local Noetherian $\cO$-algebra which is
  the universal deformation ring for~$\cG_n$-valued deformations
  of~$\rbar$ which have multiplier~$\varepsilon^{1-n}\delta_{F/F^+}^n$, are
  unramified outside~$p$, and whose restrictions to the places above~$p$ are
  crystalline and ordinary with Hodge--Tate weights
  $0,1,\dots,n-1$. The association $\psi\mapsto r_{\psi}$ induces a
  homomorphism $R_F\to R$, which is easily checked to be a
  surjection. (Indeed,
  it suffices to show that the map $R_F\to R$ induces a surjection on reduced
  cotangent spaces. It in turn suffices to see that the induced map
  of Lie algebras from~\eqref{eqn:Gn-to-cGn} is a split injection of
  $G_{\Q}$-representations, or equivalently (since $p>2$) a split
  injection of $G_F$-representations, which is clear.)
  The polarized representation 
$(\overline{r}|_{G_F},\overline{\mu})$ is ordinarily automorphic by~\cite[Thm.\
A]{symmetric} applied to~$f$ (together with quadratic base change), and the group
$\overline{r}(G_{F(\zeta_p)})$ is adequate by Lemma~\ref{lem: adequacy}.
Applying~\cite[Thm.\ 10.1]{jack}, we see that $R_F$ is a finite $\cO$-algebra (see~\cite[Thm.\ 2.4.2]{BLGGT} for a restatement in the precise form
we use here; in the notation of that statement, we are taking $l=p$, $n=p-1$,
$S=\{p\}$, $\mu=\varepsilon^{1-n}$, $H_{\tau}=\{0,1,\dots,n-1\}$).
Thus~$R$ is a finite
$\cO$-algebra, and since it has dimension at least~$1$, it has
a~$\Qpbar$-valued point. The corresponding lift
$\rho:G_{\Q}\to G_n(\Qpbar)$ of~$\rhobar$ is unramified
outside~$p$, has multiplier~$\varepsilon^{1-n}$, and is crystalline
 and ordinary with Hodge--Tate weights
$0,1,\dots,n-1$.

The representation~$\rho$ is automorphic by~\cite[Thm.\
2.4.1]{BLGGT} (taking~$F=\Q$, $l=p$, $n=p-1$, $r=\rho$,
and~$\mu=\varepsilon^{1-n}\delta_{F/F^+}^n$). More
precisely, there is a self-dual regular algebraic cuspidal automorphic
representation~$\pi$ of~$\GL_{n}(\A_{\Q})$ whose corresponding
$p$-adic Galois representation $\rho_{\pi}:G_{\Q}\to\GL_{n}(\Qpbar)$ is
isomorphic to~$\rho$. By local-global compatibility (e.g.\
\cite[Thm.\ 2.1.1]{BLGGT}) 
we see that~$\pi$ has level one and weight
zero, as claimed. \end{proof}

\begin{lem}
  \label{lem: adequacy}Let $p>5$ 
  and let~$\rbar:G_{\Q}\to\GL_2(\Fpbar)$ be a
  representation with
  $\SL_2(\Fp)\subseteq\rbar(G_{\Q})$. Then for $p-2\le n\le p$, the
  group~$(\Sym^{n-1}\rbar)(G_{\Q(\zeta_p)})$ is
  adequate in the sense of~\cite[Defn.\ 2.20]{MR3598803}. %
\end{lem}
\begin{proof}
  Since~$\SL_2(\Fp)$ is perfect, we have
  $\SL_2(\Fp)\subseteq\rbar(G_{\Q(\zeta_p)})$, so it follows from Dickson's
  classification that for some power~$q$
  of~$p$, we have $\SL_2(\F_q)\subseteq \rbar(G_{\Q(\zeta_p)})$, and
  $p\nmid [\rbar(G_{\Q(\zeta_p)}):\SL_2(\F_q)]$. By~\cite[Rem.\ 6.1]{MR3626555},  it suffices to check that
  for~$U$ the standard $2$-dimensional $\Fpbar$-representation
  of~$G=\SL_2(\F_q)$, $V:=\Sym^{n-1}U$ is adequate. It is absolutely
  irreducible (because $n\le p$), and is therefore adequate
  by~\cite[Cor.\ 9.4]{MR3626555}, noting that since $p>5$ we have $n\ge p-2>(p+1)/2$.
\end{proof}

\subsection{The case $p=107$}\label{subsec: p 107} We now prove the cases~$n=105$ and~$n=106$ of~Theorem~\ref{main}
as an application of Theorem~\ref{thm: congruence for f p and k}. %

\begin{thm}
  \label{thm: p 107 results}There exist self-dual cuspidal automorphic
representations~$\pi$ for $\GL_n/\Q$ of level one and weight zero
for~$n=105$ and~$n=106$. In particular, $H^*_{\cusp}(\GL_n(\Z),\C) \ne
0$ for these~$n$.
   \end{thm}

\begin{proof}
  Let~$f = \Delta E^2_4 E_6 = q - 48 q^2 - 195804 q^3  + \ldots $
be the unique normalized cuspidal Hecke eigenform for~$\SL_2(\Z)$ of
weight~$k=26$.  Let~$p=107$, and
~$\rhobar: G_{\Q} \rightarrow \GL_2(\F_{107})$ denote the mod~$107$ Galois
representation associated to~$f$ (in its cohomological
normalization). By~\cite[Cor., p.SwD-31]{SerreAntwerp}, the image
of~$\rhobar$ is exactly~$\GL_2(\F_{107})$ (note
that~$(\F_{107}^{\times})^{25} = \F^{\times}_{107}$). Since
$$a_{107}(f) = 35830422465487817813321292 \equiv -1 \bmod 107,$$ $f$ is
ordinary at~$107$.

Certainly~$(106,25)=1$, so in view of Theorem~\ref{thm: congruence
  for f p and k} %
we only need to check that~$\rhobar_f|_{G_{\Qp}}$ is semisimple.
That this is indeed the case is a consequence of a computation of
Elkies, recorded in~\cite[\S17]{Gross}: the form~$f$ admits a
\emph{companion form} of weight~$p+1-k = 82$, i.e.\ an eigenform~$g$
of level one and weight~$82$ with $\rhobar_f\cong
\varepsilonbar^{-25}\rhobar_g$. The
semisimplicity of~$\rhobar_f|_{G_{\Qp}}$ is an immediate consequence
of the existence of~$g$ (see e.g.\ \cite[Prop.\ 13.8(3)]{Gross}). %
By  
Theorem~\ref{thm: congruence for f p and k} %
we
deduce the existence of the desired automorphic forms~$\pi$ for~$\GL_n/\Q$ %
for~$n=105,106$
respectively. The existence of such~$\pi$ implies the non-vanishing of
the cuspidal cohomology groups (see Remark~\ref{rem:GLn-Matsushima}). %
\end{proof}

   \begin{remark} \label{forcomment} Combining Theorem~\ref{thm: p 107 results} with the
     descent result~\cite[Thm.\ 7.2]{MR2075885}, we see that
     there is a globally generic, non-endoscopic, cuspidal automorphic
   representation for $\Sp_{104}/\Q$ of level one and weight zero. If~$\AAA_g$ is  the moduli space of principally polarized abelian varieties of dimension~$g$, we deduce %
   that~$H^*_{\cusp}(\AAA_{52},\C) \ne 0$.
   However, as
   Olivier Ta\"{\i}bi explained to us, one can construct  cuspidal
   cohomology classes
of~$\AAA_g$ for much smaller~$g$ coming from endoscopic representations, and one can
even arrange that these endoscopic representations are tempered; see
~\cite[\S \ 1.24]{ChevRenard} for a closely related discussion.
\end{remark}

\section{The non-ordinary case}\label{sec: nonordinary} 
We now explain how to improve~$n=105$ to~$n=79$, at the cost of a
slightly more involved construction. The
idea behind the proof is again quite simple: we replace the ordinary eigenform~$f$
in Theorem~\ref{thm: congruence for f p and k} by a non-ordinary
form, where one can hope to use the change of weight results
of~\cite{BLGGT}. It turns out that there is no local obstruction to the existence
of a weight zero lift of (a twist of) $\Sym^{n-1}\rhobar_f$ if $n-1=p-1$
or~$p$. However, in the latter case the global representation
$\Sym^{n-1}\rhobar_f$ is reducible, and we do not know whether to
expect a congruence to exist in level one, while in the former case it has dimension
~$p$, which is excluded by the
hypotheses of~\cite{BLGGT}. Nonetheless, in the case~$n-1=p-1$, we are able
to use a simplified version of the 
arguments  of~\cite{BLGGT}, since we do not need to change
the level and only need to make a relatively simple change of weight,
and indeed our arguments are very close to those of~\cite{blgg}. 
\begin{thm}\label{thm: congruence for f p and k non ordinary}Let~$p>5$ be
  a prime, and let~$f$ be an eigenform of level~$\SL_2(\Z)$ and
  weight~$2\le k <p$, %
  such that:
  \begin{enumerate}
  \item $(k-1,p+1)=1$.
  \item $f$ is non-ordinary at~$p$.
   \end{enumerate}
  Then there exists a self-dual cuspidal automorphic
  representation~$\pi$ for $\GL_{p}/\Q$ of level one and weight zero
  whose mod~$p$ Galois representation
  $\rhobar_{\pi}:G_{\Q}\to\GL_{p}(\Fpbar)$ is isomorphic to
  $\Sym^{p-1}\rhobar_{f}$.  \end{thm}
  \begin{proof}Where possible, we follow the proof of Theorem~\ref{thm: congruence for
      f p and k}. We begin by showing that~$\rhobar_f$ has image containing~$\SL_2(\F_p)$. Since~$(k-1,p+1) = 1$, the projective image
 of $\rhobar_f(G_{I_{\Qp}})$  contains a cyclic subgroup of order~$p+1 > 5$, so~$\rhobar_f$ does not have exceptional image
 (that is, projective image~$A_4$, $S_4$, or~$A_5$).
 Since~$\rhobar_f|_{G_{\Qp}}$ is absolutely irreducible, so is~$\rhobar_f$. Hence it remains to
   rule out the possibility that~$\rhobar_f$ has dihedral image. If
   this were the case, then since
   it is unramified outside~$p$, it would have to be induced
   from~$\Q(\sqrt{p^*})$ where~$p^* = (-1)^{(p-1)/2}p$. 
   But this would imply that~$\rhobar_f |_{G_{\Q_{p}}}$ is induced
   from~$\Q_p(\sqrt{p^*})$, which would in turn imply that it is invariant under twisting by~$\varepsilon^{(p-1)/2}=\omega^{(p^2-1)/2}_2$.
   Since~$\rhobar_f |_{I_p} \simeq \omega^{k-1}_2 \oplus \omega^{p(k-1)}_2$, this can only
   happen if~$k \equiv (p+3)/2 \pmod{p+1}$, contradicting the assumption that~$(k-1,p+1) = 1$.
    
      Let~$\F/\Fp$ be a finite extension such that
  $\rhobar_f(G_{\Q})\subseteq \GL_2(\F)$, and write
  $\rhobar:=\Sym^{p-1}\rhobar_{f}$, so that
  $\rhobar:G_{\Q}\to\GO_{p}(\F)$ has multiplier 
  ~$\varepsilonbar^{1-p}=1$, and~$\rhobar(G_{\Q(\zeta_p)})$ is adequate by Lemma~\ref{lem: adequacy}.

 Let $\varepsilon_2, \varepsilon_2' : G_{\Q_{p^2}} \to
 \overline{\Z}_p^\times$ be the two Lubin--Tate characters trivial on
 $\operatorname{Art}_{\Q_{p^2}}(p)$, and write~$\omega_2$ for the reduction
 modulo~$p$ of~$\varepsilon_2$. For any $n,m\ge 1$ we let $\rho_{n,m}$
 denote the
 representation
  \[\Sym^{n-1}\Ind_{G_{\Q_{p^2}}}^{G_{\Qp}}\varepsilon_2^{-m}:G_{\Qp}\to\GL_{n}(\overline{\Z}_p),\]which
 is crystalline with Hodge--Tate weights $0,m,\dots,(n-1)m$.

We have  \[\rhobar_{p,m}\cong
 \varepsilonbar^{m(p-1)/2}\oplus\bigoplus_{i=1}^{(p-1)/2}\Ind_{G_{\Q_{p^2}}}^{G_{\Qp}}\omega_2^{m(1-p)i}.\]
Suppose that~$(m,p+1)=1$ (so that in particular~$m$ is odd). Then~$\omega^{m(1-p)}_2$ has order exactly~$p+1$,
and the~$\Gal(\Q_{p^2}/\Q_p)$-conjugate of~$\omega_2^{m(1-p)i}$ is~$\omega_2^{-m(1-p)i}$.
It follows, under this assumption on~$m$, that~$\rhobar_{p,m}$ does not depend on~$m$, so
 there is an isomorphism of orthogonal representations
 $\rhobar_{p,m}\cong \rhobar_{p,1}$. Our assumptions
 that~$f$ is non-ordinary, that~$k<p$, and that~$(k-1,p+1)=1$ therefore  imply that
  $\rhobar|_{G_{\Qp}}\cong \rhobar_{p,1}$, which admits the weight~$0$
  crystalline lift~$\rho_{p,1}$.

  Write~$R$ for the complete local Noetherian $\cO$-algebra which is
  the universal deformation ring for~$\GO_p$-valued deformations
  of~$\rhobar$ which have multiplier~$\varepsilon^{1-p}$, are
  unramified outside~$p$, and whose restrictions to~$G_{\Qp}$ are
  crystalline of weight~$0$, and lie on the same component of the
  corresponding local crystalline deformation ring
  as~$\rho_{p,1}$. By~\cite[Prop.\ 4.2.6]{Bellovin}, every irreducible
  component of~ $R$ has Krull dimension at least~$1$ (the verification
  that ~$H^0(\Q,\mathfrak{so}^*_p(1)) = 0$ is exactly as in the proof
  of Theorem~\ref{thm: congruence for f p and k}).

  Let~$F^+/\Q$ and $F/F^+$ be quadratic extensions, with~$F^+$ real
  quadratic and~$F$ imaginary
  CM, such that~$p$ is inert
  in~$F^+$, the place of~$F^+$ above~$p$ splits in~$F$, and~$F/\Q$ is disjoint from
  $(\Qbar)^{\ker\rhobar}(\zeta_{p})$. As in the proof of~\cite[Prop.\
  4.1.1]{BLGGT}, using~\cite[Cor.\ A.2.3, Lem.\ A.2.5]{BLGGT} we can
  find a cyclic CM extension~$M/F$ of degree~$(k-1)$, and characters
  $\theta,\theta':G_{M}\to\overline{\Q}_p^{\times}$ with
  $\overline{\theta}=\overline{\theta}'$, such that the
  representation $\sbar:=\Ind_{G_M}^{G_F}(\overline{\theta}\otimes\rhobar|_{G_F})$
  is absolutely irreducible. Furthermore we choose~$\theta,\theta'$ so
  that~$\theta\theta^c=\varepsilon^{2-k}$,
  $\theta'(\theta')^c=\varepsilon^{p(2-k)}$,
  and~$\Ind_{G_M}^{G_F}\theta,\Ind_{G_M}^{G_F}\theta'$ are both  
  crystalline, with all sets of labelled Hodge--Tate weights
  respectively equal to~$\{0,1,\dots,k-2\}$,
  $\{0,p,\dots,p(k-2)\}$.

By construction, after possibly replacing~$F^{+}$ by a solvable extension, we can and
  do assume that for each place~$v|p$ of~$F$, we
  have \[\bigl(\Ind_{G_M}^{G_F}\theta\bigr)|_{G_{F_v}}\sim
    \rho_{k-1,1}|_{G_{F_v}},\  \bigl(\Ind_{G_M}^{G_F}\theta'\bigr)|_{G_{F_v}}\sim
    \rho_{k-1,p}|_{G_{F_v}},\]  where $\sim$ is the notion ``connects
  to'' of~\cite[\S1.4]{BLGGT}.
   We
  let $R_F$ be the complete local Noetherian $\cO$-algebra which is
  the universal deformation ring for~$\cG_{(k-1)p}$-valued
  deformations of (the usual extension of)~$\sbar$, which have
  multiplier~$\varepsilon^{1-(k-1)p}\delta_{F/F^+}$, are unramified
  outside~$p$, and whose restrictions to the places above~$p$ are
  crystalline with Hodge--Tate weights $0,1,\dots,(k-1)p-1$, and lie
  on the same irreducible components of the local crystalline
  deformation rings as \[
    (\rho_{k-1,p}\otimes\rho_{p,1})|_{G_{F_v}}\cong
    \rho_{(k-1)p,1}|_{G_{F_v}}\cong (\rho_{p,k-1}\otimes\rho_{k-1,1})|_{G_{F_v}}.\]
     We have a finite map $R_F\to R$, taking a lifting~$\rho$
     of~$\rhobar$ to $\Ind_{G_M}^{G_F}(\theta\otimes\rho|_{G_F})$.

     We
  claim that the conclusions of ~\cite[Prop.\ 7.2]{MR3598803} apply in
  our setting, so that 
  $R_F$ is a finite $\cO$-algebra by~\cite[Thm.\ 10.1]{jack}. Admitting this claim for a moment,
  we deduce that~$R$ is a
  finite $\cO$-algebra, and since it has dimension at least~$1$, it
  has a~$\Qpbar$-valued point. The corresponding lift
  $\rho:G_{\Q}\to \GO_p(\Qpbar)$ of~$\rhobar$ is unramified
  outside~$p$, has multiplier~$\varepsilon^{1-p}$, and is crystalline
  with Hodge--Tate weights $0,1,\dots,p-1$. By~ \cite[Thm.\ 7.1]{MR3598803},
  $\Ind_{G_M}^{G_F}(\theta\otimes\rho|_{G_F})$ is automorphic,
  so~$\rho$ itself is automorphic by~\cite[Lem.\ 2.2.1, 2.2.2,
  2.2.4]{BLGGT}.

  It remains to show that we can apply \cite[Thm.\ 7.1, Prop.\
  7.2]{MR3598803}. To this end, we note that the notion of adequacy in
  ~\cite[Defn.\ 2.20]{MR3598803} can be relaxed to assume only
  that~$H^1(H,\operatorname{ad})=0$, rather than assuming that~$H^1(H,\operatorname{ad}_0)=0$;
  more precisely, the proof of~\cite[Prop.\ 2.21]{MR3598803} only uses
  this weaker assumption. Now, since $\rhobar(G_{\Q(\zeta_p)})$ is
  adequate, and since $p\nmid (k-1)$, we see that
  $\sbar(G_{F(\zeta_p)})$ is adequate by~\cite[Lem.\ A.3.1]{blggu2}
  (whose proof goes over unchanged in this setting), as required.  
   \end{proof}

 \begin{cor}
   \label{cor:79}There exists a self-dual cuspidal automorphic
  representation~$\pi$ for $\GL_{79}/\Q$ of level one and weight zero.
 \end{cor}
\begin{proof} There exists (\cite{MR1824885,Ghitza})  a modular eigenform~$f$ of level~$1$ and weight~$k=38$
   which is non-ordinary at~$p=79$, and~$(37,79+1) = 1$.
 \end{proof}

  \begin{remark}\label{rem: expect positive density irreducible} The prime~$p=79$ is the second smallest prime for which there exists a non-ordinary form~$f$ of weight~$k < p$.
The smallest is~$p=59$ for which there exists a non-ordinary eigenform of weight~$k=16$. However, $(k-1,p+1) \ne 1$ in this case,
so the construction fails in a number of
places. Following~\cite{Ghitza}, we see 
that there exist modular forms~$f$
satisfying the hypotheses of 
Theorem~\ref{thm: congruence for f p and k non ordinary}
for~$p=79,151,173,193,\ldots$
and 
modular forms satisfying the
hypotheses of Theorem~\ref{thm: congruence for f p and k}  for~$p = 107, 139, 151, 173, 179, \ldots$. We expect (but have no idea how to prove)
 that (in either case) there exist such~$f$ for a positive density  of primes~$p$. %
  \end{remark}

\bibliographystyle{amsalpha}
\bibliography{Weightzero}

\end{document}